\documentclass[leqno]{amsart}
\usepackage{amsmath,amsfonts,amsthm,amssymb,indentfirst,epic,url,centernot}

\newtheorem{theorem}{Theorem}
\newtheorem{lemma}[theorem]{Lemma}
\newtheorem{corollary}[theorem]{Corollary}

\newtheorem{question}[theorem]{Question}
\newtheorem{example}[theorem]{Example}
\newtheorem{conjecture}[theorem]{Conjecture}
\newtheorem{convention}[theorem]{Conventions}

\newcommand{\Q}{\mathbb Q}
\newcommand{\Z}{\mathbb Z}
\renewcommand{\r}{\mathrm}
\newcommand{\p}{^{(p)}}
\newcommand{\card}{\mathrm{card}}

\begin{document}

\begin{center}
\texttt{Comments, corrections, and related references welcomed,
as always!}\\[.5em]
{\TeX}ed \today
\vspace{2em}
\end{center}

\title{Thoughts on Eggert's Conjecture}
\thanks{This preprint is readable online at
\url{http://math.berkeley.edu/~gbergman/papers/}
and \url{http://arxiv.org/abs/1206.0326}\,.
The former version is likely to be updated more frequently than
the latter.
}


\author{George M. Bergman}
\address{University of California\\
Berkeley, CA 94720-3840, USA}
\email{gbergman@math.berkeley.edu}

\begin{abstract}
Eggert's Conjecture says that if $R$ is a finite-dimensional
nilpotent commutative algebra over a perfect field $F$ of
characteristic $p,$ and $R\p$ is the image of the $\!p\!$-th
power map on $R,$ then $\dim_F R\geq p\,\dim_F R\p.$
Whether this very elementary statement is true is not known.

We examine heuristic evidence for this conjecture,
versions of the conjecture that are not limited
to positive characteristic and/or to commutative $R,$
consequences the conjecture
would have for semigroups, and examples
that give equality in the conjectured inequality.

We pose several related questions, and briefly survey
the literature on the subject.
\end{abstract}
\maketitle

\section{Introduction}\label{S.intro}

If $F$ is a field of characteristic $p,$ and $R$ is a commutative
$\!F\!$-algebra, then the set $R\p$ of $\!p\!$-th powers
of elements of $R$ is not only closed under multiplication,
but also under addition, by the well-known identity
\begin{equation}\label{d.+^p}
(x+y)^p\ =\ x^p+y^p\quad (x,y\in R).
\end{equation}
Hence $R\p$ is a subring of $R.$
If, moreover, $F$ is a \emph{perfect} field (meaning that every
element of $F$ is a $\!p\!$-th power -- as is true, in
particular, if $F$ is finite, or, at the other
extreme, algebraically closed), then
the subring $R\p$ is also closed under multiplication by
elements of $F:$
\begin{equation}\label{d.ax^p}
a\,x^p\ =\ (a^{1/p}x)^p\in R\p\quad (a\in F,\,x\in R).
\end{equation}
In this situation we can ask ``how big'' the subalgebra $R\p$ is
compared with the algebra $R,$ say in terms of dimension over $F.$

If we take for $R$ a polynomial algebra $F[x]$ over a perfect
field $F,$ we see that $R\p=F[x^p],$
so intuitively, $R\p$ has a basis consisting of one out of every $p$
of the basis elements of $R.$
Of course, these bases are infinite, so we can't divide the
cardinality of one by that of the other.
But if we form finite-dimensional
truncations of this algebra, letting $R=F[x]/(x^{N+1})$ for large
integers $N,$ then we see that the dimension of $R\p$
is indeed about $1/p$ times the dimension of $R.$
If we do similar constructions starting
with polynomials in $d$ variables,
we get $R\p$ having dimension about $1/p^d$ times that of $R.$

Is the ratio $\dim R\p/\dim R$ always small?
No; a trivial counterexample is $R=F;$ a wider class of
examples is given
by the group algebras $R=F\,G$ of finite abelian groups $G$ of orders
relatively prime to~$p.$
In $G,$ every element is a $\!p\!$-th power, hence $R\p$ contains
all elements of $G,$ hence, being closed under addition
and under multiplication by members of $F,$ it is all of $R;$ so
again $\dim\,R\p/\dim\,R=1.$

In the above examples, the $\!p\!$-th power map eventually ``carried
things back to themselves''.
A way to keep this from happening is to assume the algebra $R$ is
\emph{nilpotent}, i.e., that for some positive integer $n,$
$R^n=0,$ where $R^n$ denotes the space of all sums of $\!n\!$-fold
products of members of $R.$
This leads us to

\begin{conjecture}[Eggert's Conjecture \cite{E}]\label{Cj.E}
If $R$ is a finite-dimensional nilpotent
commutative algebra over a perfect field
$F$ of characteristic $p>0,$ then
\begin{equation}\label{d.E}
\dim_F\,R\ \geq\ p\,\dim_F\,R\p.
\end{equation}
\end{conjecture}

Of course, a nonzero nilpotent algebra does not have a unit.
Readers who like their algebras unital may think of the $R$
occurring above and throughout this note as the maximal
ideal of a finite-dimensional local unital $\!F\!$-algebra.

Let us set down some conventions.

\begin{convention}\label{Cv.cm_ass}
Throughout this note, $F$ will be a field.
The symbol ``$\dim\!$'' will always stand for ``$\dim_F\!$'',
i.e., dimension as an $\!F\!$-vector-space.

Except where the contrary is stated \textup{(}in a few brief
remarks and two examples\textup{)}, $\!F\!$-algebras will be assumed
associative, but not, in general, unital.
\textup{(}Most of the time, we will be considering commutative
algebras, but we will make commutativity explicit.
When we simply write ``associative algebra'',
this will signal ``not necessarily commutative''.\textup{)}
An {\em ideal} of an $\!F\!$-algebra will mean a ring-theoretic
ideal which is also an $\!F\!$-subspace.

If $R$ is an $\!F\!$-algebra, $V$ an $\!F\!$-subspace
of $R,$ and $n$ a positive integer, then $V^n$ will denote the
\emph{\mbox{$\!F\!$-subspace}} of $R$ spanned
by all $\!n\!$-fold products of
elements of $V,$ while $V^{(n)}$ will denote the \emph{set} of
$\!n\!$-th powers of elements of $V.$
\end{convention}

Thus, if $V$ is a subspace of a commutative
algebra $R$ over a perfect field $F$ of characteristic $p,$ then
$V\p$ will also be a subspace of $R,$ but for a general
base-field $F,$ or for noncommutative $R,$ this is not so.
The map $x\mapsto x^p$ on a commutative algebra $R$ over a field of
characteristic $p$ is called the \emph{Frobenius map}.

We remark that the unital rings $R=F[x]/(x^{N+1})$ that we discussed
before bringing in nilpotence
generally fail to satisfy~(\ref{d.E}).
Most obvious is the case $N=0,$ where $R=F.$
More generally, writing $N=pk+r$
$(0\leq r< p),$ so that the lowest and highest powers of $x$
in the natural basis of $R$ are $x^0$ and $x^{pk+r},$
we find that $\dim\,R\p/\dim\,R=(k+1)/(pk+r+1),$
which is $>1/p$ unless $r=p-1.$

The corresponding nilpotent algebras are constructed
from the ``nonunital polynomial algebra'', i.e., the algebra of
polynomials with zero constant term, which we shall write
\begin{equation}\label{d.[F][x]}
[F][x]\ =\ \{\,\sum_{i>0} a_i x^i\,\}\ \subseteq\ F[x].
\end{equation}
When we divide this by the ideal generated by $x^{N+1},$
again with $N=pk+r$ $(0\leq r<p),$ we find that
$\dim\,R\p/\dim\,R=k/(pk+r),$ which is always $\leq 1/p,$
with equality only when $r=0,$ i.e., when $p\,|\,n.$

As before, examples like $R=[F][x,y]/(x^M,\,y^N)$
give ratios $\dim\,R\p/\dim\,R$ strictly lower than $1/p.$
This suggests that generation by more than one element tends to
lower that ratio, and that perhaps that ratio can equal $1/p$ only
for cyclic algebras.
This is not the case, however.
Indeed, it is easy to verify that that ratio is multiplicative
with respect to tensor products,
\begin{equation}\label{d.OX}
\dim\,(R\otimes S)\p/\dim\,(R\otimes S)\ =
\ (\dim\,R\p/\dim\,R)\,(\dim\,S\p/\dim\,S).
\end{equation}
Hence if we
tensor a nilpotent algebra $R$ of the form $[F][x]/(x^{pk+1}),$
for which we have seen that the ratio is $1/p,$ with a non-nilpotent
algebra for which the ratio is $1$
(for instance, a group algebra $F\,G$ with $p\centernot{|}\,|G|),$
we get further nilpotent examples for which the ratio is $1/p.$
Also, $\dim\,R$ and $\dim\,R\p$ are both additive
with respect to \emph{direct products}; so any direct product of two
nilpotent algebras for each of which the ratio is $1/p$ is another
such algebra.
In \S\ref{S.eg} we will discover further examples in which
the ratio comes out exactly $1/p,$ for reasons that are less clear.

\section{A first try at proving Eggert's Conjecture}\label{S.V^n}

We have seen that for $R$ a commutative
algebra over a perfect field $F$ of characteristic $p>0,$
the $\!p\!$-th power map on
$R$ over $F$ is ``almost'' linear.
In particular, its image is a vector subspace (in fact, a subalgebra).

Pleasantly, we can even find a vector subspace $V\subseteq R$
which that map sends bijectively to $R\p.$
Namely, take any $\!F\!$-basis $B$ for $R\p,$ let $B'$
be a set consisting of exactly one $\!p\!$-th root of
each element of $B,$ and let $V$ be the $\!F\!$-subspace
of $R$ spanned by $B'.$
Since the $\!p\!$-th power map sends $\sum_{x\in B'} a_x x$
to $\sum_{x\in B'} a_x^p x^p,$ it hits each
element of $R\p$ exactly once.

This suggests the following approach to Eggert's Conjecture.
Suppose we take such a subspace $V,$ and look at the subspaces
$V,\ V^2,\ \dots,\ V^p$
(defined as in the last paragraph of Convention~\ref{Cv.cm_ass}).
Can we deduce that each of them has dimension at least
that of $V\p=R\p$ (as the first and last certainly do),
and conclude that their sum within $R$ has dimension at least $p$
times that of $R\p$?

The answer is that yes, we can show that each has
dimension at least that of $R\p,$ but no, except under
special additional hypotheses, we cannot say that
the dimension of their sum is the sum of their dimensions.

The first of these claims can be proved in a context that does
not require positive characteristic, or commutativity, or nilpotence.
We \emph{will} have to assume $F$ algebraically closed; but we will
subsequently see that for commutative algebras over a perfect field
of positive characteristic, the general case reduces to that case.

\begin{lemma}\label{L.dimV^i}
Let $F$ be an algebraically closed field, $R$ an associative
$\!F\!$-algebra, $V$ a finite-dimensional subspace of $R,$ and $n$
a positive integer such that every nonzero element of $V$ has
nonzero $\!n\!$-th power.
Then for all positive integers $i\leq n$ we have
\begin{equation}\label{d.dimV^i}
\dim\,V^i\ \geq\ \dim\,V.
\end{equation}
\end{lemma}

\begin{proof}
Let $d=\dim\,V,$ and let $x_1,\dots,x_d$ be a basis for $V$ over $F.$
Suppose, by way of contradiction, that for some $i\leq n$
we had $\dim\,V^i=e<d.$
Then we claim that some nonzero $v\in V$ must satisfy $v^i=0.$

Indeed, writing the general element of $V$
as $v=a_1 x_1 + \dots + a_d x_d$ $(a_1,\dots,a_d\in F),$ we see
that the condition $v^i=0,$ expressed in terms of an
$\!e\!$-element basis of $V^i,$ consists of $e<d$ equations,
each homogeneous of positive degree (in fact, all of the
same degree, $i),$ in $d$ unknowns $a_1,\dots,a_d.$
But a system of homogeneous polynomial equations of positive
degrees with fewer equations
than unknowns over an algebraically closed field always has
a nontrivial solution \cite[p.65, Corollary~3*]{Mumford};
so, as claimed, there is a nonzero $v\in V$ with $v^i=0.$

Multiplying by $v^{n-i}$ if $i<n,$ or leaving the equation unchanged
if $i=n,$ we see that $v^n=0,$ contradicting
the hypothesis on $V,$ and completing the proof.
\end{proof}

(We could even have generalized the above proof to nonassociative
algebras, if we defined $x^i$ inductively as, say, the right-bracketed
product $x(x(\dots x)),$ and $V^i$ similarly as $V(V(\dots V)).)$

Now if $F$ is any perfect field of characteristic $p,$ and $n=p$
(or more generally, a power of $p),$ and $R$ is commutative, then
the $\!n\!$-th power map is, up to adjustment of scalars,
a linear map of $\!F\!$-vector-spaces,
so the statement that it sends no nonzero
element of $V$ to $0$ says it has
trivial kernel; and this property is preserved under extension
of scalars to the algebraic closure of
$F,$ as are the dimensions of the various spaces $V^i.$
Hence, as stated earlier, in this situation Lemma~\ref{L.dimV^i}
implies the corresponding result with ``algebraically closed''
weakened to ``perfect''.

But unfortunately, we cannot say that
$\dim\,R\geq\sum_{i\leq p}\dim\,V^i$
unless we know that the sum of the $V^i$ is direct.
Here is a special case in which the latter condition clearly holds.

\begin{corollary}\label{C.graded}
Let $R$ be a finite-dimensional commutative algebra over a perfect
field $F$ of characteristic $p>0,$ and assume that $R$ is graded by
the positive integers, is generated by its homogeneous
component $R_1$ of degree $1,$ and satisfies $(R_2)\p=0.$

Then $\dim\,R_1,\,\dots,\,\dim\,R_p$ are all $\geq\dim\,R\p,$ so
$\dim\,R\geq p\,\dim\,R\p.$
\end{corollary}

\begin{proof}
Since $R$ is the direct sum of its subspaces $R_i,$ its subalgebra
$R\p$ will be the direct sum of its subspaces $(R_i)\p\subseteq R_{ip}.$
Since $R$ is generated by $R_1,$ we have $R_{i+1}=R_i\,R_1$ for
all $i;$ hence $(R_{i+1})\p=(R_i)\p\,(R_1)\p;$ hence as
$(R_2)\p$ is zero, so are $(R_3)\p,\ (R_4)\p,\ \cdots\,.$
Hence $R\p=(R_1)\p.$

Now let $d=\dim\,R\p\subseteq R_p,$ and
take a $\!d\!$-dimensional subspace $V\subseteq R_1$ such that the
$\!p\!$-th power map carries $V$ bijectively to $R\p.$
By Lemma~\ref{L.dimV^i} and the discussion following it,
we have $\dim\,V^i\geq d$ for $i=1,\dots,p,$ hence
\begin{equation}\label{d.sum}
\dim\,R\ =\ \sum_{i=1}^\infty\dim\,R_i\ \geq
\ \sum_{i=1}^p\dim\,V^i\ \geq\ p\,d\ =\ p\,(\dim\,R\p).\qedhere
\end{equation}
\end{proof}

One might hope to get a similar result for ungraded nilpotent $R,$
by taking the filtration $R\supseteq R^2\supseteq R^3\supseteq\dots,$
and studying the associated graded algebra,
$S=\bigoplus_i S_i$ with $S_i=R^i/R^{i+1}.$
This will indeed be generated by $S_1;$ but unfortunately, $R\p$ will
not in general be embedded in $S_p,$ since an element that can
be written as a $\!p\!$-th power of one element may be
expressible in another way as
a product of more than $p$ factors (or a sum of such products), in
which case it will have zero image in $S_p=R^p/R^{p+1}.$
(What one can easily deduce by this approach is that
$\dim R\geq p\,\dim(R\p/(R\p\cap R^{p+1})).$
But that is much weaker than Eggert's conjecture.)

Putting aside the question of whether we
can reduce the ungraded case to the graded, let us ask whether,
assuming $R$ graded and generated by $R_1,$ we can weaken the
hypothesis $(R_2)\p=0$ of Lemma~\ref{L.dimV^i}.
Suppose we instead assume $(R_3)\p=0.$
Thus, $R\p=(R_1)\p\oplus (R_2)\p\subseteq R_p\oplus R_{2p}.$

In addition to our subspace $V\subseteq R_1$ which is
mapped bijectively to $(R_1)\p$ by the $\!p\!$-th power map,
we can now choose a subspace $W\subseteq R_2$ that is
mapped bijectively to $(R_2)\p.$
Letting $d_1=\dim\,(R_1)\p=\dim\,V$ and $d_2=\dim\,(R_2)\p=\dim\,W,$
we can deduce from Lemma~\ref{L.dimV^i} that
$\dim\,R_1,\ \dim\,R_2,\ \dots\,,\ \dim\,R_p$ are all $\geq\,d_1$
and that
$\dim\,R_2,\ \dim\,R_4,\ \dots\,,\ \dim\,R_{2p}$ are all $\geq\,d_2.$
The trouble is, these two lists overlap in
$\{R_2,\ R_4,\ \dots\,,\ R_{2\lfloor p/2\rfloor}\},$
while we know nothing about the sizes of the $R_i$ for odd $i$
between $p+1$ and $2p.$
If we could prove that they, like the $R_i$ for even $i$ in that
range, all had dimensions at least $d_2,$ we would be in good shape:
With $R_i$ at least $\!d_1\!$-dimensional for $i=1,\dots,p$
and at least $\!d_2\!$-dimensional for $i=p+1,\dots,2p,$ we would have
total dimension at least
$p\,\dim(R_1)\p+p\,\dim(R_2)\p=p\,\dim\,R\p.$

One might imagine that since $\dim\,R_i$ is at least
$\dim\,R_2\p$ for all even $i\leq 2p,$ those dimensions could
not perversely come out smaller for $i$ odd.
However, the following example, though involving a
noncommutative ring, challenges this intuition.

\begin{example}\label{E.a(xy)^nb}
For any positive integer $d$ and any field $F,$ there exists
an associative graded $\!F\!$-algebra
$R,$ generated by $R_1,$ such that the dimension of the
component $R_n$ is $2d$ for every odd $n>2,$
but is~$d^2+1$ for every even $n>2.$
\end{example}

\begin{proof}[Construction]
Let $R$ be presented by $d+1$ generators
$x,\ y,\ z_1,\ \dots,\ z_{d-1}$ of degree $1,$ subject to the
relations saying that $xx=yy=0,$ and that every $\!3\!$-letter
word in the generators that does \emph{not} contain the
substring $xy$ is likewise $0.$
It is easy to verify that the nonzero words of length $>2$ are
precisely those strings consisting of a ``core'' $(xy)^m$ for
some $m\geq 1,$ possibly preceded by an arbitrary letter other than $x,$
and/or followed by an arbitrary letter other than $y.$
One can deduce that for $m\geq 1,$ the nonzero words of odd
length $2m+1$
are of two forms, $(xy)^m\,a$ and $a\,(xy)^m$ for some letter $a,$
and that for each of these forms there
are $d$ choices for $a,$ giving $2d$ words altogether; while
for words of even length $2m+2$ there are also
two forms, $a\,(xy)^m\,b$ and $(xy)^{m+1},$ leading to $d^2+1$ words.
\end{proof}

Even for commutative $R,$ we can get a certain amount
of irregular behavior:

\begin{example}\label{E.4343}
For any field $F$ there exists a commutative graded $\!F\!$-algebra $R,$
generated by $R_1,$ such that the dimensions of $R_1,$
$R_2,$ $R_3,$ $R_4$ are respectively $4,$ $3,$ $4,$ $3.$
\end{example}

\begin{proof}[Construction]
First, let $S$ be the commutative algebra presented by
generators $x,\,y,\,z_1,\,z_2$ in degree $1,$ and relations saying that
$z_1$ and $z_2$ have zero product with all four generators.
We see that for all $n>1$ we have $\dim S_n=n+1,$ as in the polynomial
ring $[F][x,y],$
so $S_1,\dots,S_4$ have dimensions $4,$ $3,$ $4,$ $5.$
If we now impose an arbitrary pair of independent relations homogeneous
of degree $4,$ we get a graded algebra $R$ whose
dimension in that degree is $3$ rather than $5,$
without changing the dimensions in lower degrees.
\end{proof}

As we shall note in \S\ref{S.literature}, much of the work
towards proving Eggert's Conjecture in the literature has
involved showing that such misbehavior in the sequence of
dimensions is, in fact, restricted.

(Incidentally, if we take $F$ in Example~\ref{E.4343} to be perfect
of characteristic~$3,$ and divide out by $R_4,$
we do not get a counterexample to Eggert's Conjecture;
rather, $(R_1)^{(3)}$ turns out to be a proper subspace of $R_3.)$

\section{Relations with semigroups}\label{S.semigp}

The examples we began with in \S\ref{S.intro} were ``essentially''
semigroup algebras of abelian semigroups.

To make this precise, recall that a \emph{zero} element in a semigroup
$S$ means an element $z$ (necessarily unique) such that
$sz=zs=z$ for all $s\in S.$
If $S$ is a semigroup with zero, and $F$ a field, then the
\emph{contracted semigroup algebra} of $S,$ denoted
$F_0\,S,$ is the $\!F\!$-algebra with basis $S-\{z\},$ and
multiplication which agrees on this basis with the multiplication
of $S$ whenever the latter gives nonzero values, while when the
product of two elements of $S-\{z\}$ is $z$ in $S,$
it is taken to be $0$ in this algebra \cite[\S5.2, p.160]{C+P}.
So, for example, the algebra $[F][x]/(x^{N+1})$
of \S\ref{S.intro} is the contracted semigroup
algebra of the semigroup-with-zero presented as
such by one generator $x,$ and the one relation $x^{N+1}=z.$
(Calling this a presentation as a semigroup-with-zero means
that we also assume the relations making the products of
all elements with $z$ equal to~$z.)$

Above (following \cite{C+P})
I have written $z$ rather than $0$ in $S,$ so as to be able to talk
clearly about the relationship between the
zero element of $S$ and that of $F_0\,S.$
But since these are identified in the construction of the latter
algebra, we shall, for the remainder of this section,
write $0$ for both, as noted in

\begin{convention}\label{Cv.semigp}
In this section, semigroups with zero will be written multiplicatively,
and their zero elements written $0.$

If $X$ is a subset of a semigroup $S$ \textup{(}with or
without zero\textup{)} and $n$ a
positive integer, then $X^n$ will denote the set
of all $\!n\!$-fold products of
elements of $X,$ while $X^{(n)}$ will denote the set of all
$\!n\!$-th powers of elements of~$X.$
A semigroup $S$ with zero will be called \emph{nilpotent} if
$S^n=\{0\}$ for some positive integer $n.$
\end{convention}

Clearly, $F_0\,S$ is nilpotent as an algebra if and only
if $S$ is nilpotent as a semigroup.

If we could prove Eggert's Conjecture, I claim that we could deduce

\begin{conjecture}[semigroup version of Eggert's Conjecture]\label{Cj.E_semigp}
If $S$ is a finite nilpotent commutative semigroup with
zero, then for every positive integer $n,$
\begin{equation}\label{d.E_semigp}
\card(S\,{-}\,\{0\})\ \geq\ n\ \card(S^{(n)}{-}\,\{0\}).
\end{equation}
\end{conjecture}

Let us prove the asserted implication:

\begin{lemma}\label{L.alg=>semigp}
If Conjecture~\ref{Cj.E} is true, then so is
Conjecture~\ref{Cj.E_semigp}.
\end{lemma}

\begin{proof}
Observe that for any two
positive integers $n_1$ and $n_2,$ and any semigroup $S,$
we have $(S^{(n_1)})^{(n_2)}=S^{(n_1\,n_2)}.$
Hence, given $n_1$ and $n_2,$ if~(\ref{d.E_semigp}) holds for
all semigroups $S$ whenever $n$ is taken to be $n_1$ or $n_2,$
then it is also true for all $S$ whenever $n$ is taken to be $n_1 n_2.$
Indeed, in that situation we have
\begin{equation}\label{d.()()...}
\card(S\,{-}\,\{0\})\ \geq
\ n_1\ \card(S^{(n_1)}{-}\,\{0\})\ \geq
\ n_1 n_2\ \card(S^{(n_1\,n_2)}{-}\,\{0\}).
\end{equation}

Since~(\ref{d.E_semigp}) is trivial for $n=1,$ it will therefore
suffice to establish~(\ref{d.E_semigp}) when $n$ is a prime $p.$
In that case, let $F$ be any perfect field of characteristic $p.$
From~(\ref{d.+^p}) and~(\ref{d.ax^p}) we see that
$(F_0\,S)\p=F_0(S\p),$ and by construction,
$\dim_F\,F_0\,S=\card(S-\{0\}).$
Applying Conjecture~\ref{Cj.E} to $F_0\,S,$ we
thus get~(\ref{d.E_semigp}) for $n=p,$ as required.
\end{proof}

A strange proof, since to obtain the result for an $n$
with $k$ distinct prime factors,
we must work successively with semigroup algebras over
$k$ different fields!

So much for what we \emph{could} prove if we knew Eggert's Conjecture.
What can we conclude about semigroups using what we \emph{have} proved?
By the same trick of passing to contracted semigroup
algebras, Lemma~\ref{L.dimV^i} yields

\begin{corollary}[to Lemma~\ref{L.dimV^i}]\label{C.X^i}
Let $S$ be a commutative semigroup with zero, let $p$ be a prime, and
let $X$ be a finite subset of $S$ such that the $\!p\!$-th power
map is one-to-one on $X,$ and takes no nonzero element of $X$ to $0.$
Then
\begin{equation}\label{d.X^i}
\card(X^i{-}\,\{0\})\ \geq\ \card(X{-}\,\{0\})\quad
\mbox{ \ for \ }1\leq i\leq p.\hspace*{1.57in}\qed\hspace*{-1.57in}
\end{equation}
\end{corollary}

Note that even though Lemma~\ref{L.dimV^i} was proved for
not necessarily commutative $R$ and for exponentiation
by an arbitrary integer $n,$ we have to assume
in Corollary~\ref{C.X^i} that
$S$ is commutative and $p$ a prime, in order to call
on~(\ref{d.+^p}) and conclude that $(F_0\,X)\p=F_0\,(X\p),$

(Incidentally, the same proof gives us the
corresponding result for semigroups $S$
without zero, with~(\ref{d.X^i}) simplified by removal of
the two ``$-\,\{0\}$''s.
However, this result is an immediate consequence of
the present form of Corollary~\ref{C.X^i},
since given any semigroup $S$ and subset $X\subseteq S,$
we can apply that corollary to
$X$ within the semigroup with zero $S\cup\{0\};$
and in that case, the symbols ``$-\,\{0\}$'' in~(\ref{d.X^i})
have no effect, and may be dropped.
Inversely, a proof of Corollary~\ref{C.X^i} from the version for
semigroups without zero is possible, though not as straightforward.)

I see no way of proving the analog of Corollary~\ref{C.X^i}
with a general integer $n$ replacing the prime $p.$
(One \emph{can} get it for prime-power values, by noting
that~(\ref{d.+^p}) and hence Lemma~\ref{L.dimV^i} work
for exponentiation by $p^k.$
I have not so stated those results only for simplicity of presentation.)
We make this

\begin{question}\label{Q.X^i}
Let $S$ be a commutative semigroup with zero, let $n$ be a positive
integer, and let $X$ be
a finite subset of $S$ such that the $\!n\!$-th power
map is one-to-one on $X,$ and takes no nonzero element of $X$ to $0.$
Must $\card(X^i{-}\,\{0\})\geq\card(X{-}\,\{0\})$
for $1\leq i\leq n$?
\end{question}

\section{Some plausible and some impossible generalizations}\label{S.genlzn}

The hypothesis on Corollary~\ref{C.X^i} concerns $\card(X\p-\{0\}),$
while the conclusion is about $\card(X^i-\{0\}).$
It is natural to ask whether we can make the hypothesis and the
conclusion more parallel, either by replacing $X^i$
by $X^{(i)}$ in the latter (in which case the inequality
in the analog of~(\ref{d.X^i}) would become equality, since
$X^{(i)}-\{0\}$ can't be larger than $X-\{0\}),$ or by
replacing $X\p$ by $X^p$ in the former.

But both of these generalizations are false, as shown by
the next two examples.

\begin{example}\label{E.X^(q)}
For any prime $p>2,$ and any $i$ with $1<i<p,$ there exists
a commutative semigroup $S$ with zero, and a subset $X$
such that the $\!p\!$-th power map is one-to-one on $X$
and does not take any nonzero element of $X$ to $0,$
but such that $\card(X^{(i)}{-}\,\{0\})<\card(X{-}\,\{0\}).$
\end{example}

\begin{proof}[Construction]
Given $p$ and $i,$ form the direct
product of the nilpotent semigroup $\{x,\,x^2,\dots,x^p,\,0\}$
and the cyclic group $\{1,y,\dots,y^{i-1}\}$ of order $i,$ and let
$X$ be the subset $\{x\}\times\{1,y,\dots,y^{i-1}\}.$
Then the $\!p\!$-th power map from $X$ to $X\p$ (which
is also $X^p)$ is bijective, the common
cardinality of these sets being $i;$ but
$X^{(i)}=\{x^i\}\times\{1\}$ has cardinality~$1.$
To make this construction a semigroup with zero, we may identify the
ideal $\{0\}\times\{1,y,\dots,y^{i-1}\}$ to a single element.
\end{proof}

\begin{example}\label{E.X^p}
For any prime $p>2$ there exist a commutative
semigroup $S$ with zero, and a subset $X\subseteq S,$ such
that $\card(X-\{0\})=\card(X^p-\{0\}),$ but such that for
all $i$ with $1<i<p,$ $\card(X^i-\{0\})<\card(X-\{0\}).$

\end{example}

\begin{proof}[Construction]
Let $S$ be the abelian semigroup with zero presented by $p+1$
generators, $x,\,y,\,z_1,\dots,z_{p-1},$ and relations
saying that each $z_i$ has zero product with every generator
(including itself).
Thus, $S$ consists of the elements
of the free abelian semigroup on $x$ and $y,$
together with the $p$ elements $0,\,z_1,\dots,z_{p-1}.$

Let $X$ be our generating set $\{x,\,y,\,z_1,\dots,z_{p-1}\}.$
Then we see that for every $i>1,$ the set $X^i-\{0\}$ has
$i+1$ elements, $x^i,\,x^{i-1}y,\dots,y^i.$
Hence $\card(X^p-\{0\})=p+1=\card(X-\{0\});$ but for $1<i<p,$
$\card(X^i-\{0\})=i+1<p+1.$
(We can make this semigroup finite by setting every 
member of $X^{p+1}$ equal to $0.)$
\end{proof}

(In the above examples, the case $p=2$ was excluded because
in that case, there are no $i$ with $1<i<p.$
However, one has the corresponding constructions with any prime
power $p^r>2$ in place of $p,$ including powers of $2,$ as long as
one adds to the statement corresponding to Example~\ref{E.X^(q)} the
condition that $i$ be relatively prime to $p.)$\vspace{.5em}

From the construction of Example~\ref{E.X^(q)}, we can also obtain a
counterexample to a statement which, if it were true, would,
with the help of Lemma~\ref{L.dimV^i},
lead to an easy affirmative answer to Question~\ref{Q.X^i}:

\begin{example}\label{E.F0X^n}
There exists a commutative semigroup $S$ with zero, a finite
subset $X\subseteq S,$ and an integer $n>0,$
such that the $\!n\!$-th power map is one-to-one on $X$
and does not take any nonzero element of $X$ to $0,$
but such that for some field $F,$ the $\!n\!$-th power map
on the span $FX$ of $X$ in $F_0 S$
does take some nonzero element to~$0.$
\end{example}

\begin{proof}[Construction]
Let us first note that though we assumed in
Example~\ref{E.X^(q)} that $p$ was a prime to emphasize
the relationship with Corollary~\ref{C.X^i}, all we needed
was that $p$ and $i$ be relatively prime.
For the present example, let us repeat that construction
with any integer $n>2$ (possibly, but not necessarily, prime)
in place of the $p$ of that construction,
while using a prime $p<n,$ not dividing $n,$
in place of our earlier $i.$
Thus, the $\!n\!$-th power map is one-to-one on $X,$
but the $\!p\!$-th power map is not.

Now let $F$ be any algebraically closed field of characteristic $p.$
Then on the subspace $FX\subseteq F_0 S,$
the $\!p\!$-th power map is (up to adjustment of scalars) an
$\!F\!$-linear map to the space $F X\p$ of smaller dimension; hence
it has nontrivial kernel.
(For the particular construction used in
Example~\ref{E.X^(q)}, that kernel contains  $x-xy.)$
But any element annihilated by the $\!p\!$-th power
map necessarily also has $\!n\!$-th power~$0.$
\end{proof}

The use of a field $F$ of positive characteristic
in the above construction suggests the following question,
an affirmative answer to which would indeed,
with Lemma~\ref{L.dimV^i},
imply an affirmative answer to Question~\ref{Q.X^i}.

\begin{question}\label{Q.F0X^n}
Suppose $X$ is a finite subset of a commutative semigroup
$S$ with zero, $n$ a positive integer
such that the $\!n\!$-th power map is one-to-one on $X$
and does not take any nonzero element of $X$ to $0,$
and $F$ a field of characteristic~$0.$
Must every nonzero element of the span $FX$ of $X$ in $F_0 S$
have nonzero $\!n\!$-th power?
\end{question}

In a different direction, Lemma~\ref{L.dimV^i}
leads us to wonder whether
there may be generalizations of Eggert's Conjecture independent
of the characteristic.

As a first try, we might consider a nilpotent commutative algebra $R$
over any field $F,$ and for an arbitrary positive integer $n$
ask whether $\dim(\r{span}(R^{(n)}))/\dim\,R\leq 1/n,$
where $\r{span}(R^{(n)})$ denotes the $\!F\!$-subspace
of $R$ spanned by $R^{(n)}.$
But this is nowhere near true.
Indeed, I claim that
\begin{equation}\label{d.spanR^(n)}
\mbox{If the characteristic of $F$ is either $0$ or $>n,$
then $\r{span}(R^{(n)})=R^n.$}
\end{equation}
For it is not hard to verify that for any $x_1,\dots,x_n\in R,$
\begin{equation}\label{d.n!prod}
\sum_{S\subseteq\{1,\dots,n\}}
(-1)^{\r{card}(S)}\ (\sum_{i\in S} x_i)^n\ =
\ (-1)^n\,n!\ x_1\dots\,x_n.
\end{equation}
(Every monomial of degree $n$ in $x_1,\dots,x_n$
other than $x_1\dots x_n$
fails to involve some $x_m,$ hence the sets indexing
summands of~(\ref{d.n!prod}) in which that monomial
appears can be paired off, $S\leftrightarrow S\cup\{m\},$
one of even and one of odd cardinality.
Hence the coefficients of every such monomial cancel, leaving only the
multiple of the monomial $x_1\dots x_n$ coming from $S=\{1,\dots,n\}.)$
Under the assumption on the characteristic of $F$
in~(\ref{d.spanR^(n)}), $n!$ is invertible,
so~(\ref{d.n!prod}) shows that $x_1\dots x_n\in \r{span}(R^{(n)}),$
proving~(\ref{d.spanR^(n)}).
Now taking $R=[F][x]/(x^{N+1})$ for $N\geq n,$
we see that $\r{span}(R^{(n)})=R^n$
has basis $\{x^n,\,x^{n+1},\dots,\,x^N\};$ so
$\dim(\r{span}(R^{(n)}))/\dim\,R=(N-n+1)/N,$ which
for large $N$ is close to $1,$ not to $1/n.$

However, something nearer to the spirit of Lemma~\ref{L.dimV^i},
with a chance of having a positive answer, is

\begin{question}\label{Q.V->}
Let $R$ be a finite-dimensional nilpotent commutative algebra
over an algebraically closed field $F,$ let $V$ be a subspace
of $R,$ and let $n$ be a positive integer such that every
nonzero element of $V$ has nonzero $\!n\!$-th power.
Must $\dim\,R\geq n\,\dim\,V$?
\end{question}

Above, $V$ is a subspace of $R,$
but in the absence of~(\ref{d.+^p}), we can't expect
$V^{(n)}$ to simultaneously be one.
In the next question, we turn the tables, and make the
\emph{target} of the $\!n\!$-th power map a subspace.

\begin{question}\label{Q.->W}
Let $R$ be a finite-dimensional nilpotent commutative algebra
over an algebraically closed field $F,$ let $W$ be a subspace
of $R,$ and let $n$ be a positive integer such that every
element of $W$ is an $\!n\!$-th power in $R.$
Must $\dim\,R\geq n\,\dim\,W$?
\end{question}

Let us look at the above two questions for $R=[F][x]/(x^{N+1}).$
Note that an element $r\neq 0$ of this algebra has
$r^n \neq 0$ if and only if the lowest-degree term of $r$ has degree
$\leq N/n,$ while a necessary
condition for $r$ to {\em be} an $\!n\!$-th power
(which is also sufficient if $n$ is not divisible by the
characteristic of $F)$ is that its lowest-degree
term have degree divisible by $n.$
Now for each of these properties, there are, in general,
large-dimensional {\em affine} subspaces of $[F][x]/(x^{N+1})$
all of whose elements have that property.
E.g., if $n\leq N,$ the $\!(N{-}1)\!$-dimensional affine space of
elements of the form $x+(\mbox{higher degree terms})$
consists of elements whose $\!n\!$-th powers
are nonzero, and for $F$ of characteristic not
divisible by $n,$ the $\!(N{-}n)\!$-dimensional affine space of
elements of the form $x^n+(\mbox{higher degree terms})$ consists
of $\!n\!$-th powers.
In each of these cases, if we fix $n$ and let $N\to\infty,$
the ratio of the dimension of our affine subspace
to that of our algebra approaches~$1.$
But these affine subspaces are not vector subspaces!
If $U$ is a vector subspace of $R=[F][x]/(x^{N+1}),$
and if for each $x^m$ which appears as the lowest degree term
of a member of $U,$ we choose a $w_m\in U$ with that lowest degree
term, it is not hard to see that these elements form a basis of $U.$
It is easily deduced from the above
discussion that if $U$ consists of $\!n\!$-th powers, or
consists of elements which, if nonzero, have nonzero $\!n\!$-th power,
then $U$ has dimension $\leq N/n.$
So for this $R,$ Questions~\ref{Q.V->}
and~\ref{Q.->W} both have affirmative answers.

Can those two questions be made the $m=1$ and $m=n$
cases of a question statable for all $1\leq m\leq n$?
Yes.
The formulation is less elegant than for those two
cases, but I include it for completeness.

\begin{question}\label{Q.U}
Let $R$ be a finite-dimensional nilpotent commutative algebra
over an algebraically closed field $F,$ let $U$ be a subspace
of $R,$ and let $1\leq m\leq n$ be integers such that every nonzero
element of $U$ has an $\!m\!$-th root in $R$ whose $\!n\!$-th
power is nonzero.
Must $\dim\,R\geq n\,\dim\,U$?
\end{question}

(Again, we easily obtain an affirmative answer
for $R=[F][x]/(x^{N+1}),$ essentially as in the cases $m=1$ and $m=n.)$

Early on, in thinking about Eggert's Conjecture, I convinced myself
that the noncommutative analog was false.
But the analog I considered was based on replacing $R\p$ by
$\r{span}(R\p)$ so that one could talk about its dimension.
However, the generalizations considered
in Questions~\ref{Q.V->}-\ref{Q.U}
are also plausible for noncommutative rings.

I also assumed in Questions~\ref{Q.V->}-\ref{Q.U} that
$F$ was algebraically closed,
because that hypothesis was essential to the proof of
Lemma~\ref{L.dimV^i}, and is the condition under which
solution-sets of algebraic equations behave nicely.
However, I don't have examples showing
that the results asked for are false without it.
So let us be bold, and ask

\begin{question}\label{Q.rash}
Does the generalization of Conjecture~\ref{Cj.E_semigp},
or an affirmative answer to any of Questions~\ref{Q.X^i},
\ref{Q.F0X^n}, \ref{Q.V->}, \ref{Q.->W} or \ref{Q.U}, hold
if the commutativity hypothesis is dropped, and/or, in the case of
Question~\ref{Q.V->}, if the assumption that $F$ be algebraically closed
is dropped \textup{(}or perhaps weakened to ``$F$
is infinite''\textup{)}?

\textup{(}For Question~\ref{Q.->W} one can
similarly drop the assumption that $F$ be
algebraically closed; but then one would want to
change the hypothesis that every element of $W$ have an
$\!n\!$-th root to the condition, equivalent thereto in the
algebraically closed case, that every $\!1\!$-dimensional
subspace of $W$ contain a nonzero $\!n\!$-th power, since the original
hypothesis would be unreasonably strong over
non-algebraically-closed $F.$
One can likewise make the analogous generalization
of Question~\ref{Q.U}.\textup{)}
\end{question}

If we go further, and drop not only the characteristic~$p$
assumption and the algebraic closedness of $F,$
but also the \emph{associativity} of $R,$ then there is an
easy counterexample to the analog of Eggert's Conjecture.

\begin{example}\label{E.nonassoc}
For every positive integer $d,$
there exists a graded, nilpotent, commutative, \emph{nonassociative}
algebra over the field $\mathbb{R}$ of real numbers,
$R=R_1\oplus R_2\oplus R_3,$
generated by $R_1,$ in which the respective dimensions of the
three homogeneous components are $d,\ 1,\ d,$ and in which
the ``cubing'' operation $r\mapsto r(rr)$
gives a bijection from $R_1$ to $R_3.$

Hence, writing $R^{(3)}$ for $\{r(rr)\mid r\in R\}=R_3,$ we
have $\dim R^{(3)}/\dim R=d/(2d+1),$ which is $>1/3$ if $d>1.$
\end{example}

\begin{proof}[Construction]
Let $W$ be a real inner product space of dimension $d,$
let $A=W\oplus\,\mathbb{R},$ made an $\!\mathbb{R}\!$-algebra
by letting elements of $\mathbb{R}\subseteq A$ act on $A$ on either
side by scalar multiplication, and letting the product of two elements
of $W$ be their inner product in $\mathbb{R}.$
Note that $W,\ W^2,\ W^3$ are respectively $W,\ \mathbb{R},\ W,$
and that on $W,$ the operation $w\mapsto w(ww)$ takes every element
to itself times the square of its norm, hence is a bijection $W\to W.$

For the above $A,$ let
us form $A\,\otimes_\mathbb{R}\,[\mathbb{R}][x]/(x^4),$ which is
clearly nilpotent; let $V$ be its subspace $Wx;$ and let $R$ be
the subalgebra generated by $V;$ namely,
$(Wx)\oplus(\mathbb{R}\,x^2)\oplus(Wx^3).$
Then the asserted properties are clear.
\end{proof}

The parenthetical comment following Lemma~\ref{L.dimV^i} shows,
however, that over an algebraically closed base field $F,$ there
is no example with the corresponding properties.

If in Example~\ref{E.nonassoc} we let $B$ be an orthonormal basis
of $W,$ then on closing $Bx\subseteq R$ under the
multiplication of $R$ (but not under addition or scalar multiplication),
we get a $\!2d+2\!$-element structure (a ``nonassociative
semigroup'', often called a ``magma'') which is a
counterexample to the nonassociative analogs of
Conjecture~\ref{Cj.E_semigp}, Corollary~\ref{C.X^i}
and Question~\ref{Q.X^i}.

I will end this section by recording,
for completeness, a positive-characteristic version of
Example~\ref{E.nonassoc} (though the characteristic will
not be the exponent whose behavior the example involves).
Before stating it, let us recall that a nonassociative
algebra is called \emph{power-associative} if every
$\!1\!$-generator subalgebra is associative; equivalently,
if the closure of every singleton $\{x\}$ under the multiplication
(intuitively, the set of ``powers'' of $x)$ is in fact a semigroup.
Let us call a graded nonassociative algebra
\emph{homogeneous-power-associative} if the subalgebra generated by
every homogeneous element is associative.
Example~\ref{E.nonassoc} above is easily seen to be
homogeneous-power-associative.
The same property in the next example will allow us to
avoid having to specify the
bracketing of the power operation we refer to.

\begin{example}\label{E.nonassoc_p}
For every prime $p,$
there exists a graded, nilpotent, commutative, nonassociative,
but homogeneous-power-associative
algebra $R=R_1\oplus \dots\oplus R_{p+1}$ over
a non-perfect field $F$ of characteristic $p,$
such that $R$ is generated by $R_1,$
the $\!p\,{+}1\!$-st power operation gives a surjection
$R_1\to R_{p+1}$ taking no nonzero element to zero,
and $\dim R_i=p$ for $i<p$ and for $i=p+1,$ but $\dim R_p=1.$

Hence, $\dim R^{(p+1)}/\dim R=p/(p^2+1)>1/(p+1).$
\end{example}

\begin{proof}[Sketch of construction]
Given $p,$ let $F$ be any field of characteristic $p$
having a proper purely inseparable extension $F'=F(u^{1/p}),$
such that every element of $F'$ has a $\!p\,{+}1\!$-st root in $F'.$
(We can get such $F$ and $F'$ starting with any
algebraically closed field $k$ of characteristic $p,$ and
any subgroup $G$ of the additive group $\Q$ of rational numbers
which is $\!p\,{+}1\!$-divisible but not $\!p\!$-divisible.
Note that $p^{-1}G\subseteq\Q$ will have the form
$G+p^{-1}h\,\Z$ for any $h\in G - p\,G.$
Take a group isomorphic to $G$ but written multiplicatively, $t^G,$ and
its overgroup $t^{p^{-1}G},$ and
let $F$ and $F'$ be the Mal'cev-Neumann power series
fields $k((t^G))$ and $k((t^{p^{-1}G}))$
\cite[\S2.4]{PMC_SF}, \cite{GMB_HL};
and let $u\in F$ be the element $t^h.$
The asserted properties are easily verified.)

Let us now form the (commutative, associative)
truncated polynomial algebra $[F'][x]/(x^{p+2}),$
graded by degree in $x,$
and let $R$ be the $\!F\!$-subspace of this algebra consisting of
those elements for which the coefficient of $x^p$ lies
in the subfield $F$ of $F'$ (all other coefficients
being unrestricted).
We make $R$ a graded nonassociative $\!F\!$-algebra by using the
multiplication of $[F'][x]/(x^{p+2})$ on all pairs of homogeneous
components {\em except} those having degrees summing to $p,$ while
defining the multiplication when the degrees sum to $p$ by fixing
an $\!F\!$-linear retraction $\psi: F'\to F,$ and taking the product of
$a\,x^i$ and $b\,x^{p-i}$ $(0<i<p,$ $a,b\in F')$ to be $\psi(ab)\,x^p.$

We claim that $R$ is homogeneous-power-associative; in fact, that powers
of homogeneous elements of $R,$ however bracketed, agree with
the values of
these same powers in the associative algebra $[F'][x]/(x^{p+2}).$
Note first that the evaluations of powers of elements homogeneous
of degrees other than $1$ never pass through $R_p,$ so they certainly
come out as in $[F'][x]/(x^{p+2}).$
For an element $a\,x$ of degree $1$ $(a\in F'),$ the
same reasoning holds for powers less than the $\!p\!$-th.
In the case of the $\!p\!$-th power,
the last stage in the evaluation of any bracketing of $(a\,x)^p$
takes the form $(a\,x)^i\cdot (a\,x)^{p-i}= \psi(a^i\,a^{p-i})\,x^p;$
but $a^i\,a^{p-i}=a^p\in F,$ which is fixed
by $\psi,$ so the result again comes out as in $[F'][x]/(x^{p+2}).$
Knowing this, it is easy to verify likewise that all computations of the
$\!p\,{+}1\!$-st power of $a\,x\in R_1$ agree with its
value in $[F'][x]/(x^{p+2}).$

The other asserted properties are now straightforward.
In particular the $\!p\,{+}1\!$-st power map $R_1\to R_{p+1}$
is surjective, and sends no nonzero element to $0,$
because these statements are true in $[F'][x]/(x^{p+2})$
(surjectivity holding by our assumption on $\!p\,{+}1\!$-st
roots in $F').$
\end{proof}

\section{Some attempts at counterexamples to Eggert's Conjecture for semigroups}\label{S.eg}

Since Eggert's Conjecture implies the semigroup-theoretic
Conjecture~\ref{Cj.E_semigp}, a counterexample to the
latter would disprove the former.
We saw in \S\ref{S.intro} that for certain
sorts of truncated polynomial algebras over a field $F$ of
characteristic $p,$ the ratio $\dim\,R\p/\dim\,R$ was exactly
$1/p;$ i.e., as high as Eggert's Conjecture allows.
Those algebras are contracted semigroup
algebras $F_0\,S,$ where $S$ is a semigroup with zero presented by
one generator $x$ and one relation $x^{pk+1}=0;$ so these semigroups
have equality in Conjecture~\ref{Cj.E_semigp}.
It is natural to try to see whether, by some modification of
this semigroup construction, we can push the ratio
$\card(S\p{-}\,\{0\})/\card(S{-}\,\{0\})$ just a little above~$1/p.$

In scratchwork on such examples, it is convenient
to write the infinite cyclic
semigroup not as $\{x,\,x^2,\,x^3,\,\dots\,\},$ but additively, as
$\{1,\,2,\,3,\,\dots\,\}.$
Since in additive notation, $0$ generally denotes an identity
element, it is best to denote a ``zero'' element by $\infty.$
So in this section we shall not adopt
Convention~\ref{Cv.semigp}, but follow this additive notation.
Thus, the sort of nilpotent cyclic semigroup with zero that
gives equality in the statement of Conjecture~\ref{Cj.E_semigp} is
\begin{equation}\label{d.12...}
\{1,\ 2,\ \dots\,,\ N,\ \infty\,\},\quad
\mbox{where $N$ is a multiple of }\,n.
\end{equation}

For $S$ a finite nilpotent abelian semigroup with zero,
the semigroup version of Eggert's conjecture
can be written as saying that the integer
\begin{equation}\label{d.nSn-S}
n\ \card(S^{(n)}{-}\,\{\infty\})\ -\ \card(S\,{-}\,\{\infty\})
\end{equation}
is always $\leq 0.$
(We continue to write $S^{(n)}$ for what in our
additive notation is now $\{n\,x\mid x\in S\}.)$

What kind of modifications can we apply to~(\ref{d.12...})
in the search for variant examples?
We might impose a relation; but it turns out that this
won't give anything new.
E.g., if for $i<j$ in $\{1,2,\,\dots\,,\,N\}$ we impose
on~(\ref{d.12...}) the
relation $i=j,$ then this implies $i+1=j+1,$ and so forth;
and this process eventually identifies some $h\leq N$
with an integer $>N,$ which, in~(\ref{d.12...}), equals $\infty.$
So $h$ and all integers $\geq h$ fall together with $\infty;$
and if we follow up the consequences, we eventually find that every
integer $\geq i$ is identified with $\infty.$
Thus, we get a semigroup just like~(\ref{d.12...}), but with
$i-1$ rather than $N$ as the last finite value.

So let us instead pass to a subsemigroup of~(\ref{d.12...}).
The smallest change we can make is to drop $1,$ getting the
subsemigroup generated by $2$ and $3,$ which we shall now denote $S.$
Then $\card(S-\{\infty\})$ has gone down by $1,$ pushing the
value of~(\ref{d.nSn-S}) up by $1;$ but the integer $n$
has ceased to belong to $S^{(n)},$ decreasing~(\ref{d.nSn-S}) by $n.$
So in our attempt to find a counterexample, we have ``lost ground'',
decreasing~(\ref{d.nSn-S}) from $0$ to $-n+1.$

However, now that $1\notin S,$
we can regain some ground by imposing relations.
Suppose we impose the relation that identifies $N-1$
either (a)~with $N$ or (b)~with $\infty.$
If we add any member of $S$ (loosely speaking,
any integer $\geq 2)$ to both sides of either relation,
we get $\infty=\infty,$ so no additional identifications are implied.
Since we are assuming $N$ is divisible by $n,$ the integer $N-1$ is
not; so we have again decreased the right-hand term of~(\ref{d.nSn-S}),
this time without decreasing the left-hand term; and thus brought
the total value to $-n+2.$
In particular, if $n=2,$ we have returned to the value $0;$
but not improved on it.

I have experimented with more complicated examples of
the same sort, and gotten very similar results:
I have not found one that made the value of~(\ref{d.nSn-S}) positive;
but surprisingly often, it was possible to arrange things so that
for $n=2,$ that value was $0.$
Let me show a ``typical'' example.

We start with the additive subsemigroup of the natural
numbers generated by $4$ and $5.$
I will show it by listing an initial string of the
positive integers, with the members of our subsemigroup underlined:
\begin{equation}\label{d.bold}
1\ 2\ 3\ \underline{4\ 5}\ 6\ 7\ \underline{8\ 9\ 10}\ 11
\ \underline{12\ 13\ 14\ 15\ 16\ 17\ \dots\ .}
\end{equation}

Assume this to be truncated at some large integer $N$
which is a multiple of $n,$ all larger integers being
collapsed into $\infty.$
If we combine the effects on the two terms of~(\ref{d.nSn-S})
of having dropped the six integers $1,\,2,\,3,\,6,\,7,\,11$
from~(\ref{d.12...}),
we find that, assuming $N\geq 11n,$~(\ref{d.nSn-S}) is now $6(-n+1).$

Now suppose we impose the relation $i=i+1$ for some $i$ such
that $i$ and $i+1$ both lie in~(\ref{d.bold}).
Adding $4$ and $5$ to both sides of this equation, we
get $i+4=i+5=i+6;$ adding $4$ and $5$ again we get
$i+8=i+9=i+10=i+11.$
At the next two rounds, we get strings of equalities that overlap
one another; and all subsequent strings likewise overlap.
So everything from $i+12$ on falls together with $N+1$ and
hence with $\infty;$ so we may as well assume
\begin{equation}\label{d.n+1=i+12}
N+1\ =\ i+12.
\end{equation}

What effect has imposing the relation $i=i+1$ had on~(\ref{d.nSn-S})?
The amalgamations of the three strings of integers described
decrease $\card(S{-}\{\infty\})$ by $1,$ $2$ and $3$ respectively,
so in that way, we have gained ground, bringing~(\ref{d.nSn-S})
up from $6(-n+1)$ to possibly $6(-n+2).$
But have we decreased $\card(S^{(n)}{-}\{\infty\}),$ and so
lost ground, in the process?

If $n>2,$ then even if there has been no such
loss, the value $6(-n+2)$ is negative; so let us assume $n=2.$
If we are to avoid bringing~(\ref{d.nSn-S}) below $0,$
we must make sure that none of the sets that were fused into
single elements,
\begin{equation}\label{d.i+}
\{i,\,i+1\},\quad\{i+4,\,i+5,\,i+6\},\quad\{i+8,\,i+9,\,i+10,\,i+11\},
\end{equation}
contained more than one member of $S^{(2)}.$
For the first of these sets, that is no problem; and for
the second, the desired conclusion can be achieved by taking
$i$ odd, so that of the three elements of that set, only $i+5$ is even.
For the last it is more difficult -- the set will contain two even
values, and if $i$ is large, these will both belong to $S^{(2)}.$

However, suppose we take $i$ not so large; say we choose it
so that the smaller of the two even values in that set is
the largest even integer that does \emph{not} belong to $S^{(2)}.$
That is $22,$ since $11$ is the largest integer not in~(\ref{d.bold}).
Then the above considerations show
that we do get a semigroup for which~(\ref{d.nSn-S}) is zero.

The above choice of $i$ makes $i+9$ (the smallest even value in
the last subset in~(\ref{d.i+})) equal to $22$ (the largest
even integer not in $S^{(2)}),$ so $i=13,$ so
by~(\ref{d.n+1=i+12}), $N+1=25.$

Let us write down formally the contracted semigroup algebras of
the two easier examples described earlier, and of the above example.
\begin{example}\label{E.4,5}
Let $F$ be a perfect field of characteristic~$2.$
Then the following nilpotent algebras have
equality in the inequality of Eggert's Conjecture.
\begin{equation}\label{d.2,3}
R\ =\ [F][x^2,\,x^3]\,/\,(x^{N-1}-x^{N},\,x^{N+1},\,x^{N+2})
\qquad \mbox{for every even $N>2,$}
\end{equation}
\begin{equation}\label{d.2,3'}
R\ =\ [F][x^2,\,x^3]\,/\,(x^{N-1},\,x^{N+1},\,x^{N+2})
\qquad \mbox{for every even $N>2,$}
\end{equation}
\begin{equation}\label{d.4,5}
R\ =\ [F][x^4,\,x^5]\,/\,(x^{13}-x^{14},\,x^{25},\dots,x^{28}).
\end{equation}

More precisely, in both~\textup{(\ref{d.2,3})}
and~\textup{(\ref{d.2,3'})}
$\dim\,R=N-2,$ and $\dim\,R^{(2)}=(N-2)/2,$
while in~\textup{(\ref{d.4,5})},
$\dim\,R=18,$ and $\dim\,R^{(2)}=9.$
\qed
\end{example}

Many examples behave like these.
A couple more are
\begin{equation}\label{d.more}
[F][x^2,\,x^5]/(x^{11}-x^{12},\,x^{\geq 15}),\qquad
[F][x^3,\,x^7]/(x^{13}-x^{14},\,x^{\geq 25})
\end{equation}
(where ``$x^{\geq n}$'' means ``$x^n$ and all higher powers''; though
in each case, only finitely many are needed).

Perhaps Eggert's Conjecture is true, and these examples
``run up against the wall'' that it asserts.
Or --~who knows -- perhaps if one pushed this sort of
exploration further, to homomorphic images of
semigroups generated by families of three or more
integers, and starting farther from $0,$
one would get counterexamples.

For values of $n$ greater than $2,$ I don't know any examples of
this flavor that even bring~(\ref{d.nSn-S}) as high as zero.
(But a class of examples of a different sort, which does, was noted in
the last paragraph of \S\ref{S.intro}.)

Incidentally, observe that in the semigroup-theoretic context that
led to~(\ref{d.2,3}) and~(\ref{d.2,3'}), we had the choice of
imposing either the relation $N-1=N$ or the relation $N-1=\infty.$
However, in the development that
gave~(\ref{d.4,5}), setting a semigroup element equal to $\infty$
would not have done the same job as setting two such elements equal.
If we set $i=\infty,$ then, for example, $i+4$ and $i+5$ would each
become $\infty,$ so looking at the latter two elements,
we would lose one from $S^{(2)}$ as well as one not in $S^{(2)}.$
Above, we instead set $i=i+1,$ and the resulting
pair of equalities $i+4=i+5=i+6$ turned a family consisting of
two elements not in $S^{(2)}$ and one in $S^{(2)}$
into a single element of~$S^{(2)}.$

Turning back to Eggert's ring-theoretic conjecture,
it might be worthwhile to experiment with
imposing on subalgebras of $[F][x]$ relations ``close to'' those
of the sort used above, but not expressible in purely
semigroup-theoretic terms; for instance, $x^i+x^{i+1}+x^{i+2}=0,$ or
$x^i-2x^{i+1}+x^{i+2}=0.$

\section{Sketch of the literature}\label{S.literature}

The main positive results in the
literature on Eggert's Conjecture concern two kinds of
cases: where $\dim(R\p)$ (or some related
invariant) is quite small, and where $R$ is graded.

N.\,H.\,Eggert \cite{E}, after
making the conjecture, in connection with
the study of groups that can appear as the group of units of a finite
unital ring $A$ (the nonunital
ring $R$ to which the conjecture would be applied being
the Jacobson radical of $A),$ proved it for $\dim(R\p)\leq 2.$
That result was extended to $\dim(R\p)\leq 3$ by R.\,Bautista \cite{RB},
both results were re-proved more simply by
C.\,Stack \cite{CS2}, \cite{CS3}, and most recently pushed up to
$\dim(R\p)\leq 4$ by B.\,Amberg and L.\,Kazarin \cite{A+K2}.
Amberg and Kazarin also prove in \cite{A+K1} some similar
results over an arbitrary field,
in the spirit of our Questions~\ref{Q.V->} and~\ref{Q.->W},
and they show in \cite{A+K4} that, at least when the values
$\dim(R^i/R^{i+1})$ are small, these give a nonincreasing function
of $i.$
In \cite{A+K4} they give an extensive survey of results on this
subject and related group-theoretic questions.

K.\,R.\,McLean \cite{McL1}, \cite{McL2} has obtained strong positive
results in the case where $R$ is graded and generated by its
homogeneous component of degree~$1.$
In particular, in \cite{McL1} he proves Eggert's Conjecture in
that case if $(R_3)\p=0$ (recall that in Corollary~\ref{C.graded}
we could not get beyond the case $(R_2)\p=0),$
or if $R\p$ is generated by two elements.
Moreover, without either assumption (but still assuming
$R$ graded and generated in degree~$1),$ he proves
that $\dim\,R\p/\dim\,R\leq 1/(p\,{-}1).$
His technique involves taking a subspace $V\subseteq R_1$ as at the
start of \S\ref{S.V^n} above, and constructing recursively
a family of direct-sum decompositions of $V,$ each new summand arising
as a vector-space complement of the kernel of multiplication by an
element obtained using the previous steps of the construction.
He also shows in \cite{McL1} that Eggert's Conjecture
holds for the radicals of group algebras of finite
abelian groups over perfect fields $F$ of nonzero characteristic.

S.\,Kim and J.\,Park \cite{K+P} prove Eggert's Conjecture
when $R$ is a commutative nilpotent \emph{monomial algebra},
i.e., an algebra with a presentation in which all relators
are monomials in the given generators.

M.\,Korbel\'{a}\v{r} \cite{MK} has recently shown that Eggert's
Conjecture holds whenever $R\p$ can be generated as
an $\!F\!$-algebra by two elements.
(So a counterexample in the spirit of the
preceding section would require at least $3$ generators.)
\cite{MK} ends with a generalization of Eggert's conjecture,
which is equivalent to the case of Question~\ref{Q.V->} above
in which $F$ is a field of positive characteristic $p$
and $n=p,$ but $F$ is not assumed perfect.

In \cite{LH}, a full proof of Eggert's
Conjecture was claimed, but the argument was flawed.
(The claim in the erratum to that paper, that the proof is at
least valid for the graded case, is also incorrect.)

There is considerable variation in notation and language
in these papers.
E.g., what I have written $R\p$ is denoted $R^{(1)}$ in
Amberg and Kazarin's papers, $R\p$ in Stack's and Korbel\'{a}\v{r}'s,
and $R[p]$ in McLean's (modulo differences in the letter
used for the algebra $R).$
McLean, nonstandardly, takes the statement that $R$ is graded to
include the condition that it is generated by its
degree~$\!1\!$ component.

Though I do not discuss this above, I have,
also examined the behavior of the
sequence of dimensions of quotients $R^i/R^{i+1}$ for a
commutative algebra $R.$
Most of my results seem to be subsumed by those of Amberg and Kazarin,
but I will record here a question which that line of thought suggested,
which seems of independent interest for its simplicity.
Given two subspaces $V$ and $W$ of a commutative algebra, let
$\r{Ann}_V\,W$ denote the subspace $\{x\in V\mid xW=\{0\}\}\subseteq V.$

\begin{question}\label{Q.AnnV^n}
If $R$ is a commutative algebra over a
field $F,$ $V$ a finite-dimensional
subspace of $R,$ and $n$ a positive integer, must
\begin{equation}\label{d.AnnV^n}
\dim(V/\r{Ann}_V\,V^n)\ \leq\ \dim\,V^n\ \mbox{?}
\end{equation}
\end{question}

I believe I have proved~(\ref{d.AnnV^n}) for $\dim\,V^n\leq 4.$
The arguments become more intricate with each succeeding value
$1,$~$2,$~$3,$~$4.$

I am indebted to Cora Stack for bringing Eggert's Conjecture
to my attention and providing a packet of relevant literature,
to Martin Olsson for pointing me to the result in \cite{Mumford}
used in the proof of Lemma~\ref{L.dimV^i}, and to the referee for
making me justify an assertion that was not as straightforward as I
had thought.

\end{document}